\documentclass[noproofs]{rmj-public}
\usepackage[T1]{fontenc}
\usepackage{amsmath,amssymb,amsthm,mathtools,mathrsfs}
\usepackage{microtype}
\usepackage{enumitem}
\usepackage{xcolor}
\hypersetup{
    pdftitle={Analytic Properties of Self-Reproducing Families under Contour Convolution},
    pdfauthor={Gabriel Ben-Simon},
    pdfsubject={Self-reproduction, analytic extension, and finite Laurent support}
}
\renewcommand{\logobox}{}
\setlist{itemsep=3pt,topsep=5pt}
\renewcommand{\paragraph}[1]{\par\smallskip\noindent\textbf{#1}\enspace\ignorespaces}
\newtheorem{theorem}{Theorem}[section]
\newtheorem{proposition}[theorem]{Proposition}
\newtheorem{lemma}[theorem]{Lemma}
\newtheorem{corollary}[theorem]{Corollary}
\theoremstyle{definition}
\newtheorem{definition}[theorem]{Definition}

\theoremstyle{remark}

\newcommand{\C}{\mathbb C}
\newcommand{\Z}{\mathbb Z}

\newcommand{\Om}{\Omega_a}
\newcommand{\Ha}{\mathscr H_a}
\newcommand{\conv}{\mathbin{\diamond}}
\newcommand{\ind}{\operatorname{ind}}

\newcommand{\spanC}{\operatorname{span}_{\C}}

\newcommand{\dd}{\,d}

\title[Self-Reproducing Families under Convolution]{Analytic Properties of Self-Reproducing Families under Contour Convolution}
\author{Gabriel Ben-Simon}
\address{Afeka Academic College of Engineering, Tel Aviv, Israel}
\email{GavrielBe@afeka.ac.il}
\date{}
\keywords{Hadamard product; contour convolution; self-reproducing semigroups; finite Laurent support; Laurent series; idempotent projections; analytic extension.}
\subjclass{Primary 30H50; Secondary 20M14, 30B40, 46J20.}

\begin{document}
\begin{abstract}
This note grew out of two questions arising in different settings. The first concerns recognizing holomorphic functions among continuous functions on an annulus. The second concerns what self-reproducing subsets can tell us about semigroups of functions arising in analysis.

Our analytic starting point is a variation on the classical relation between holomorphy and invariance of contour integrals. It identifies the holomorphic functions as the largest class containing the constant function $1$ for which a contour convolution is independent of the chosen curve for every pair of functions. On this class, convolution amounts to multiplying Laurent coefficients.

We then study families $B$ satisfying $B\conv B=B$. We show that expressing a function as a convolution product of $m$ holomorphic functions corresponds precisely to its extension to a wider annulus determined by $m$. This identifies the largest self-reproducing family as the functions extending holomorphically to the punctured plane. We also prove that uniform boundedness on a fixed subannulus surrounding the unit circle forces all members of such a family to have Laurent expansions supported on one common finite set. A general semigroup theorem describes the minimal self-reproducing families.
\end{abstract}
\maketitle

\section{Introduction and Main Results}\label{sec:introduction}

This note grew out of two questions with rather different origins. The first arose from an earlier study, in an expository context,  of Cauchy-type integrals of continuous functions \cite{BenSimonCauchy}: how little independence of the integration path is needed to single out holomorphic functions among continuous functions on an annulus?

The second concerns commutative semigroups arising in analysis, especially semigroups of functions: what can their self-reproducing subsets tell us about their structure, particularly their idempotents, and what analytic properties might self-reproduction force?

The algebraic question arose while comparing structural features of different mathematical theories. Both directions entered the earlier investigation \cite{BenSimonCauchy}, and the present note develops part of the questions that emerged from that work.

During the preparation of this note, we also became aware of related work and several established results closely connected with our questions. We acknowledge these connections where they arise. We now describe the mathematical setting and the main conclusions, emphasizing the relations between them.

Fix $a>0$ and consider the reciprocal annulus
\[
  \Omega_a=\{z\in\C:e^{-a}<|z|<e^a\}.
\]
For $f,g\in C(\Omega_a)$, $z\in\Omega_a$, and a piecewise $C^1$ closed curve $\gamma$ of winding number $+1$ about the origin along which both $w$ and $z/w$ belong to $\Omega_a$, consider the contour operation
\begin{equation}\label{eq:introcontour}
 (f\conv_\gamma g)(z)=\frac{1}{2\pi i}\oint_\gamma
 f(z/w)g(w)\,\frac{\dd w}{w}.
\end{equation}
At this stage the functions are only continuous, and the operation may depend on the chosen curve.

We begin with the constant function $1$ and the single evaluation point $z=1$. Let $\Ha$ consist of all $f\in C(\Omega_a)$ for which
\begin{equation}\label{eq:introtest}
 (1\conv_\gamma f)(1)=\frac{1}{2\pi i}\oint_\gamma f(w)\,\frac{\dd w}{w}
\end{equation}
is independent of the winding-one curve $\gamma$ in $\Omega_a$.

This single test is sufficient. Lemma~3.2 uses integrals around polar rectangles to show that the radial Fourier coefficients are independent of the radius; coefficient estimates and Fourier uniqueness then give the Laurent expansion in Theorem~3.3. In particular,
\[
 \Ha=\mathcal O(\Omega_a).
\]
Proposition~4.2 gives the corresponding maximality statement: $\Ha$ is the largest subset of $C(\Omega_a)$ containing $1$ for which the contour products are independent of the choice of admissible curve for every pair of functions. Thus the holomorphic function space is recovered from a condition stated entirely for continuous functions.

Once this space has been obtained, Theorem~4.1 shows that the contour operation becomes the classical Hadamard product of Laurent coefficients:
\[
 c_n(f\conv g)=c_n(f)c_n(g),\qquad n\in\Z.
\]
The coefficientwise product is familiar. What matters here is that the space on which it acts has first been obtained from the contour condition, rather than assumed at the outset.

We now turn to the algebraic question. A nonempty family $B\subseteq\Ha$ is called \emph{self-reproducing} if
\[
 B\conv B=B.
\]
This requires both closure under the product and the factorization of every member into two members of the same family.

The main family-level conclusions concern analytic restrictions imposed by this condition. First, Theorem~5.2 identifies exact convolution powers with holomorphic extension to wider annuli:
\[
 \Ha^{\conv m}=\mathscr H_{ma}|_{\Omega_a}.
\]
The left side consists of actual $m$-fold products, without taking a linear span or closure. More precisely, every element of the right side is the restriction of the $m$th convolution power of a \emph{single} function in $\Ha$. Consequently, Theorem~5.3 identifies
\[
 \mathcal C_a=\mathcal O(\C^*)|_{\Omega_a},\qquad
 \mathcal C_a\conv\mathcal C_a=\mathcal C_a,
\]
as the largest self-reproducing family in $\Ha$. Every $B\subseteq\Ha$ with $B\conv B=B$ is contained in $\mathcal C_a$. Thus internal algebraic factorization forces analytic continuation from the original annulus to the punctured plane.

Second, under an additional boundedness hypothesis, a stronger conclusion holds. Theorem~7.2 states that if a self-reproducing family is uniformly bounded on one fixed two-sided subannulus, then there exists a single finite set $E\subset\Z$ such that every member of the family is a Laurent polynomial supported in $E$, with all its coefficients of modulus at most $1$. This is a common support for the entire family, which may itself be infinite; no linearity, closedness, or compactness is assumed. The proof uses Lemma~7.1, an observation about arbitrary families of sequences under coordinatewise multiplication, together with the uniform decay of Laurent coefficients.

The algebraic question also has a general answer, independent of the function-theoretic setting. Theorem~2.1, also established in \cite{BenSimonCauchy}, states that an inclusion-minimal nonempty self-reproducing subset of a commutative semigroup is exactly a singleton consisting of an idempotent. In $\Ha$, the idempotents are
\[
 P_E(z)=\sum_{n\in E}z^n,\qquad E\subset\Z\text{ finite},
\]
since their Laurent coefficients belong to $\{0,1\}$ and tend to zero at both ends. Hence the minimal self-reproducing families are exactly the singletons $\{P_E\}$. The coefficient calculation is elementary; the general semigroup theorem explains the reduction from minimal self-reproduction to idempotents. The individual powers $z^n$ are the primitive convolution idempotents and act as coefficient projections, whose finite sums reconstruct functions in $\Ha$.

The present note begins with the contour construction but concentrates on what self-reproduction requires of entire families of functions, with and without an additional boundedness assumption.

\section{Minimal self-reproduction in a commutative semigroup}\label{sec:selection}

Let $S$ be a commutative semigroup. For $B\subseteq S$, put
\[
 B^2=\{xy:x,y\in B\}.
\]
We call a nonempty $B$ \emph{self-reproducing} when
\begin{equation}\label{eq:selfreproduction}
 B^2=B.
\end{equation}
Thus the set is closed under multiplication, and every element is a product of two of its members. Here minimality means minimality under inclusion among nonempty sets satisfying \eqref{eq:selfreproduction}. 

\begin{theorem}[Minimal self-reproduction in a commutative semigroup]\label{thm:semigroup}
Let $S$ be a commutative semigroup and $B\subseteq S$ nonempty. The following are equivalent:
\begin{enumerate}[label=(\roman*)]
\item $B^2=B$, and no proper nonempty subset $C\subset B$ satisfies $C^2=C$;
\item $B=\{e\}$ for an idempotent $e\in S$.
\end{enumerate}
\end{theorem}
\begin{proof}
For $m\geq1$ put $B^{[m]}=\{b^m:b\in B\}$. By commutativity,
$(B^{[m]})^2=(B^2)^{[m]}=B^{[m]}\subseteq B$; minimality gives
$B^{[m]}=B$. Choose $b_0\in B$ and successive square roots
$b_{j+1}^2=b_j$. Then
\[
D=\{b_j^q:j\geq0,\ q\geq1\}
\]
satisfies $D\subseteq B$ and $D^2=D$: products are powers of a common
later root, and $b_j^q=(b_{j+1}^q)^2$. Thus $D=B$.
Choose $c\in B$ with $c^3=b_0$ and write $c=b_j^q$. Then
$b_j^{3q}=b_0=b_j^{2^j}$, where $3q\neq2^j$.
If two distinct positive powers of an element agree, some power is
idempotent: from $u^r=u^s$ with $r<s$, choose $k\geq r$
divisible by $s-r$, giving $u^{2k}=u^k$.
Hence $B$ contains an idempotent $e$. Its singleton is self-reproducing,
so minimality implies $B=\{e\}$. The converse is immediate.
\end{proof}
The same proof appears in \cite[Section~3.3, Lemma~3.1]{BenSimonCauchy};
for related semigroup results, see Tamura~\cite{Tamura}.

\section{From contour integrals to Laurent series}\label{sec:arena}

Set $r=e^{-a}$ and $R=e^a$, so that $rR=1$ and $\Om=\{z:r<|z|<R\}$. We first derive the Laurent series from the condition involving the constant function, and then identify $\Ha$ with the holomorphic functions.

\subsection{The condition involving the constant function}

Write $e_n(z)=z^n$ for $n\in\Z$. For $f\in C(\Om)$, $n\in\Z$, and $r<\rho<R$, define
\begin{align}
 c_n(f;\rho)
 &=\frac{1}{2\pi i}\oint_{|w|=\rho}f(w)w^{-n-1}\dd w\label{eq:coefficient}\\
 &=\frac{\rho^{-n}}{2\pi}\int_0^{2\pi}
       f(\rho e^{it})e^{-int}\dd t.\nonumber
\end{align}
All circles have positive orientation. They can also be expressed through the contour operation:
\begin{equation}\label{eq:probeidentity}
 c_n(f;\rho)=(e_n\conv_{|w|=\rho}f)(1).
\end{equation}

\begin{definition}[The condition for the constant function]\label{def:constantcoherence}
Let $\Ha$ be the space of all $f\in C(\Om)$ such that
\begin{equation}\label{eq:constantcoherence}
 (1\conv_\gamma f)(1)=\frac{1}{2\pi i}\oint_\gamma f(w)\frac{\dd w}{w}
\end{equation}
has the same value for all piecewise $C^1$ closed curves $\gamma$ of winding number $+1$ in $\Om$.
\end{definition}

The condition for the constant function determines the radial Fourier coefficients. The argument is related to that in \cite[Section~5]{CD}.

\begin{lemma}[Independence of the radial coefficients]\label{lem:radialcoherence}
For every $f\in\Ha$ and every $n\in\Z$, the value $c_n(f;\rho)$ is independent of $\rho$. Denote it by $c_n(f)$.
\end{lemma}
\begin{proof}
The integral of $f(w)\dd w/w$ around a polar rectangle in $\Om$ vanishes. To see this, append its boundary to a winding-one loop with the same basepoint. The new loop is homotopic to the original one, so \eqref{eq:constantcoherence} makes the integral around the rectangle zero. 

Write $u(\rho,t)=f(\rho e^{it})$. Fix $r<\alpha<\beta<R$ and define $J(t)=\int_\alpha^\beta u(\rho,t)\frac{\dd\rho}{\rho}$.
The rectangle with radial sides at angles $t_0,t_1$ and circular sides at radii $\alpha,\beta$ has zero integral. Parameterizing its four sides gives $J(t_1)-J(t_0)=i\int_{t_0}^{t_1}\bigl(u(\beta,t)-u(\alpha,t)\bigr)\dd t$. By the fundamental theorem of calculus, $J$ is continuously differentiable, with $J'(t)=i\bigl(u(\beta,t)-u(\alpha,t)\bigr)$.

Set $A_n(\rho)=\frac1{2\pi}\int_0^{2\pi}u(\rho,t)e^{-int}\dd t$. Integrate the identity for $J'$ against $e^{-int}$. The boundary term vanishes by periodicity, and interchanging the integrals gives $A_n(\beta)-A_n(\alpha)=n\int_\alpha^\beta A_n(\rho)\frac{\dd\rho}{\rho}$.
This implies $A_n'(\rho)=nA_n(\rho)/\rho$, hence $\rho^{-n}A_n(\rho)$ is constant. By \eqref{eq:coefficient}, this is the desired independence of $c_n(f;\rho)$.
\end{proof}

\subsection{Derived Laurent representation}

Set $M_\rho(f)=\max_{|z|=\rho}|f(z)|$. The integral definition immediately implies
\begin{equation}\label{eq:coefficientbound}
 |c_n(f)|\leq M_\rho(f)\rho^{-n}
 \qquad(r<\rho<R,\ n\in\Z).
\end{equation}

\begin{theorem}[Laurent series from the contour condition]\label{thm:representation}
Every $f\in\Ha$ has the unique representation
\begin{equation}\label{eq:laurent}
 f(z)=\sum_{n\in\Z}c_n(f)z^n,
\end{equation}
where the series converges absolutely and uniformly on each compact subannulus of $\Om$. Conversely, any Laurent series with this convergence property defines a member of $\Ha$, with the displayed coefficients.
\end{theorem}
\begin{proof}
For $n\geq0$ on a compact subannulus, \eqref{eq:coefficientbound} bounds terms by a geometric series. The same holds for negative indices, proving normal convergence. Termwise integration on $|z|=\rho$ shows that the periodic functions $f(\rho e^{it})$ and the series sum share all Fourier coefficients. By continuous Fourier uniqueness (see \cite[Chapter I, Theorem 2.7]{Katznelson}), they are equal. Conversely, absolute uniform convergence on the compact image of any allowed curve permits termwise integration of the series divided by $w$. The terms $w^{n-1}$ have global primitives for $n\ne0$, while $\oint_\gamma\dd w/w=2\pi i\ind(\gamma,0)$. Thus every winding-one integral equals $c_0$, proving the defining condition; circle integration recovers all coefficients.
\end{proof}

In particular, $\Ha=\mathcal O(\Omega_a)$. Indeed, normal convergence of the Laurent series implies holomorphy (its differentiated series converge on smaller compact subannuli). Conversely, if $f$ is holomorphic, the integral of $f(w)\dd w/w$ depends only on the winding number, so $f\in\Ha$. We equip $\Ha$ with the compact-open topology. For every $b>0$, the notation $\mathscr H_b$ denotes the same construction on $\Omega_b$.

\begin{corollary}\label{cor:decay}
The map $f\mapsto(c_n(f))_{n\in\Z}$ is injective, and $c_n(f)\longrightarrow0$ as $|n|\longrightarrow\infty$.
\end{corollary}

\section{Contour convolution and Hadamard multiplication}\label{sec:convolution}

For fixed $z\in\Om$, the allowed integration domain is
\[
\Om(z)=\{w:\max(r,|z|/R)<|w|<\min(R,|z|/r)\},
\]
which contains the unit circle. Here an admissible curve is any piecewise $C^1$ closed curve in $\Om(z)$, with arbitrary winding number. The next theorem gives the classical contour formula for Hadamard multiplication; see \cite{DS,FS}. The resulting coefficient algebra belongs to the setting of holomorphic Hadamard algebras studied by Render and Sauer~\cite{RSalg}, here with bilateral Laurent coefficients.

\begin{theorem}[Convolution and Laurent coefficients]\label{thm:product}
For $f,g\in\Ha$ and an admissible closed curve $\gamma$,
\begin{equation}\label{eq:generalproduct}
 (f\conv_\gamma g)(z)
 =\ind(\gamma,0)\sum_{n\in\Z}c_n(f)c_n(g)z^n.
\end{equation}
For a curve of winding number $+1$, the formula defines a product $f\conv g$ independent of the chosen curve, with $c_n(f\conv g)=c_n(f)c_n(g)$. The space $\Ha$ is closed under this bilinear, commutative, associative product, and $\ e_m\conv e_n=\delta_{mn}e_n$.
\end{theorem}
\begin{proof}
Along an admissible compact curve, both expansions converge absolutely and uniformly, as does their double product series. Multiplying and integrating term by term against $\frac{1}{2\pi i}\oint_\gamma w^{n-m-1}\dd w = \delta_{n,m} \ind(\gamma,0)$ gives \eqref{eq:generalproduct}. The coefficient bound on the unit circle gives $|c_n(g)|\leq M_1(g)$, so the resulting series converges normally and its sum belongs to $\Ha$.
\end{proof}

This yields the maximality statement promised in the introduction.

\begin{proposition}[Maximality of the holomorphic function space]\label{prop:maximalarena}
Suppose $V\subset C(\Om)$ contains the constant function $1$, and the contour product of every pair of elements of $V$ is unchanged under admissible homotopies. Then $V\subseteq\Ha$. Conversely, $\Ha$ contains $1$ and all the powers $e_n$, is closed under the product, and has the stated homotopy-independence property. It is therefore the largest such function space containing the constant function $1$.
\end{proposition}
\begin{proof}
For $f\in V$, apply the assumption to the pair $(1,f)$ at $z=1$. At $z=1$, reciprocity gives $\Om(1)=\Om$, and all winding-one loops in this annulus are freely homotopic. Thus Definition~\ref{def:constantcoherence} applies, forcing $f\in\Ha$. The converse is established by Theorem~\ref{thm:product}.
\end{proof}

\section{Exact factorization and analytic extension}\label{sec:factorization}

The decay estimates for Laurent coefficients show that coefficientwise products extend to wider annuli. The next results make this precise. For $0<b<a$, set
\[
 p_b(f)=\sup_{e^{-b}\leq|z|\leq e^b}|f(z)|,
 \qquad q_b(c)=\sup_{n\in\Z}|c_n|e^{b|n|}.
\]

\begin{lemma}\label{lem:weights}
A bilateral sequence $(c_n)_{n\in\Z}$ is the Laurent coefficient sequence of an element of $\Ha$ if and only if $q_b(c)<\infty$ for every $0<b<a$. Moreover,
\[
q_b(c(f))\leq p_b(f),\qquad
p_b(f)\leq q_d(c(f))\sum_{n\in\Z}e^{-(d-b)|n|}\quad(0<b<d<a).
\]
\end{lemma}
\begin{proof}
Use the circles $|z|=e^{\pm b}$ in \eqref{eq:coefficient} for the first inequality, and sum $|c_nz^n|\leq q_d(c)e^{-(d-b)|n|}$ for the second. The latter also proves the converse by Theorem~\ref{thm:representation}.
\end{proof}

The coefficientwise root argument is familiar from weighted-sequence algebras; see \cite[Lemma~7.7]{Pirk}. Here it gives the exact relation between convolution powers and annular widths.

\begin{theorem}[Exact powers]\label{thm:powers}
For every positive integer $m$, $\Ha^{\conv m}:=\{f_1\conv\cdots\conv f_m:f_j\in \Ha\} = \mathscr H_{ma}|_{\Omega_a}$. Every $F\in \mathscr H_{ma}$ has a convolution $m$th root $g\in \Ha$: $g^{\conv m}=F|_{\Omega_a}$.
\end{theorem}
\begin{proof}
For $0<t<ma$, set $b=t/m<a$. If $d_n=\prod_{j=1}^m c_n(f_j)$, then $|d_n|e^{t|n|}\leq\prod_{j=1}^m q_b(c(f_j))$, so the product extends to $\Omega_{ma}$. Conversely, for each $n$ choose any complex $m$th root $r_n$ of $c_n(F)$. For $0<b<a$, $\sup_n |r_n|e^{b|n|}=q_{mb}(c(F))^{1/m}<\infty$. The resulting Laurent series defines $g\in \Ha$ with the required coefficientwise power.
\end{proof}

This equality describes classes of restrictions; it does not say that $\Omega_{ma}$ is the maximal domain of holomorphic continuation of any individual product.

The intersection of all powers is a standard semigroup construction (see \cite{RW}). In this algebra it has the following analytic description.

\begin{theorem}[Largest self-reproducing family]\label{thm:core}
Let $\mathscr H_\infty=\mathcal O(\C^*)$. Put $\mathcal C_a=\mathscr H_\infty|_{\Omega_a}$. Then
\[
\mathcal C_a\conv\mathcal C_a=\mathcal C_a,
\qquad \bigcap_{m\geq1}\Ha^{\conv m}=\mathcal C_a.
\]
Every subset $B\subset \Ha$ satisfying $B\conv B=B$ is contained in $\mathcal C_a$. Consequently $\mathcal C_a$ is the unique largest such subset. It is proper, dense, and nonclosed in $\Ha$.
\end{theorem}
\begin{proof}
If $f\in\bigcap_{m\geq1}\Ha^{\conv m}$, Theorem~\ref{thm:powers} supplies extensions $F_m\in\mathscr H_{ma}$. Their Laurent coefficients are those of $f$, so they agree on overlaps and glue on $\bigcup_m\Omega_{ma}=\C^*$; hence $f\in\mathcal C_a$. Associativity and $B^2=B$ imply $B^m=B$, so every such $B$ is contained in this intersection. Conversely, products preserve the coefficient estimates for every width $b>0$. If $r_n^2=c_n(F)$ for $F\in\mathscr H_\infty$, then $q_b(r)=q_{2b}(c(F))^{1/2}<\infty$ for every $b>0$. These roots therefore define an element of $\mathscr H_\infty$, proving self-reproduction of $\mathcal C_a$ and its inclusion in every power. Density and properness follow from Laurent polynomials and $(e^a-z)^{-1}$, respectively.
\end{proof}

\section{Idempotents and minimal self-reproduction}\label{sec:classification}

Set $X=\Ha$ with the commutative product of Theorem~\ref{thm:product}. For finite $E\subset\Z$, define $P_E(z)=\sum_{n\in E}z^n$, with $P_\varnothing=0$. On the unit circle, $\conv$ is normalized periodic convolution and $P_E$ restricts to a classical idempotent trigonometric polynomial; see \cite{BR}. Related decompositions into primitive idempotents occur in the theory of convolution projections on compact groups \cite[Introduction]{AGST}.

\begin{proposition}\label{prop:idempotents}
The convolution idempotents of $X$ are exactly the $P_E$ with $E$ finite.
\end{proposition}
\begin{proof}
If $p\conv p=p$, coefficient multiplication implies $c_n(p)^2=c_n(p)$, so $c_n(p)\in\{0,1\}$. By Corollary~\ref{cor:decay}, only finitely many coefficients can equal $1$. Conversely, coefficient multiplication gives $P_E\conv P_E=P_E$ for every finite $E$.
\end{proof}

\begin{theorem}[Classification of minimal self-reproduction]\label{thm:minimalclassification}
A nonempty $B\subseteq X$ is inclusion-minimal among subsets satisfying $B^2=B$ if and only if
\begin{equation}\label{eq:classification}
 \boxed{\ B=\{P_E\}\quad\text{for a finite }E\subset\Z.\ }
\end{equation}
\end{theorem}
\begin{proof}
Theorem~\ref{thm:semigroup} identifies the sets as singleton idempotents, and Proposition~\ref{prop:idempotents} exactly classifies them.
\end{proof}

Minimality cannot be omitted: $\{0\}\cup\{e_n:n\in\Z\}$ is self-reproducing,
but each $\{e_n\}$ is a proper self-reproducing subset.

\section{Bounded families and common finite support}\label{sec:bounded}

\begin{lemma}[Uniform tails and finite support]\label{lem:uniformtail}
Let $I$ be an index set and $B\subset\C^I$ be nonempty, with coordinatewise multiplication and $B^2=B$. Suppose $\alpha_i=\sup_{x\in B}|x_i|<\infty$ for every $i$, and for every $\varepsilon>0$ only finitely many $i$ satisfy $\alpha_i\geq\varepsilon$. Then every $\alpha_i$ equals zero or one. The set $E=\{i:\alpha_i=1\}$ is finite, and all elements of $B$ are supported in $E$ and have coordinate moduli at most one.
\end{lemma}
\begin{proof}
Because every product of two elements lies in $B$ and every element of $B$ is such a product, $\alpha_i=\sup_{x,y\in B}|x_iy_i|=\alpha_i^2$, using independent choices of $x,y$. Since $\alpha_i$ is finite and nonnegative, it equals zero or one. Applying the uniform-tail hypothesis with $\varepsilon=1/2$ concludes the proof.
\end{proof}

\begin{theorem}[Bounded self-reproduction]\label{thm:bounded}
Let $\varnothing\ne B\subset\Ha$ satisfy $B\conv B=B$. Suppose that for some $0<b<a$ one has $\sup_{f\in B}p_b(f)<\infty$. Then there is a finite $E\subset\Z$ such that
\[
B\subset\left\{\sum_{n\in E}c_nz^n:|c_n|\leq1\right\}.
\]
\end{theorem}
\begin{proof}
Lemma~\ref{lem:weights} yields $\alpha_n=\sup_{f\in B}|c_n(f)|\leq Me^{-b|n|}$ (where $M=\max(1,\sup_{f\in B}p_b(f))$). The coefficient map respects multiplication, giving uniform tails. Lemma~\ref{lem:uniformtail} applies, and the common support is finite.
\end{proof}

\medskip
\noindent\emph{Example.} Fix $n\in\Z$ and let $B_n=\{t z^n:0<t<1\}$.
Since $(s z^n)\conv(t z^n)=st z^n$, we have $B_n\conv B_n=B_n$.
The family is uniformly bounded on every compact subannulus,
but none of its members is an idempotent; moreover,
$\sup_{f\in B_n}|c_n(f)|=1$ is not attained.

\paragraph{A weaker factorization hypothesis.}
Under the same uniform boundedness assumption, $B\subseteq B\conv B$ alone forces a common finite Laurent support. Indeed, this inclusion gives $\alpha_n\leq\alpha_n^2$; since $\alpha_n\leq Me^{-b|n|}$, it follows that $\alpha_n=0$ for all sufficiently large $|n|$. The coefficient bound $|c_n(f)|\leq1$ uses the reverse inclusion.

The two-sided width is essential. The family
$B_-:=\{0\}\cup\{z^{-k}:k=0,1,2,\ldots\}$ satisfies
$B_-\conv B_-=B_-$, since $z^{-j}\conv z^{-k}=0$ for $j\ne k$
and $z^{-k}\conv z^{-k}=z^{-k}$. It is uniformly bounded by $1$
on $1\leq|z|\leq e^b$ for every $0<b<a$, but has no common finite Laurent support.

Without uniform boundedness, self-reproducing families need not have finite support.
For example, $f_t(z)=\sum_{n\in\Z}e^{-tn^2}z^n$ ($t>0$) has infinite
Laurent support and satisfies $f_s\conv f_t=f_{s+t}$ and
$f_t=f_{t/2}\conv f_{t/2}$, whereas $f_t(1)\to\infty$ as $t\downarrow0$.

\section{Primitive idempotents, ideals, and convolution operators}\label{sec:operators}

Idempotents $p,q$ are orthogonal if $p\conv q=0$. A nonzero idempotent is primitive if it is not the sum of two nonzero orthogonal idempotents. Monomials also play a central role in the study of Hadamard algebra homomorphisms; see \cite[Theorem~1.1]{RSiso}.

\begin{theorem}[Primitive convolution idempotents]\label{thm:primitive}
The primitive convolution idempotents are exactly $e_n(z)=z^n$, $n\in\Z$. For finite $E,F\subset\Z$, $P_E\conv P_F=P_{E\cap F}$. Every idempotent has, up to order, the unique finite decomposition into orthogonal primitive idempotents $P_E=\sum_{n\in E}e_n$.
\end{theorem}

\begin{proof}
Coefficient multiplication gives the intersection formula. By Proposition~\ref{prop:idempotents}, orthogonal sums correspond to disjoint unions of finite supports; primitive supports are therefore precisely the singletons. The zero idempotent has the empty decomposition.
\end{proof}

\begin{theorem}[Finite-dimensional ideals]\label{thm:ideals}
The finite-dimensional complex linear convolution ideals of $X$ are exactly $I_E=\spanC\{e_n:n\in E\}$ for finite $E\subset\Z$. Their convolution units are $P_E$. The minimal nonzero linear ideals are exactly $\C e_n$.
\end{theorem}

\begin{proof}
For any ideal $I$, the identity $f\conv e_n=c_n(f)e_n$ shows that $I$ contains $e_n$ whenever some member has a nonzero $n$th coefficient. Finite dimension permits only finitely many such indices, giving $I=I_E$. The product rule gives the unit and minimality assertions (the zero ideal has unit $P_\varnothing=0$).
\end{proof}

Support descriptions of closed Hadamard ideals also occur in \cite{BM}; here Laurent truncation gives the classification directly.

\begin{theorem}[Closed-ideal classification]\label{thm:closedideals}
The closed complex linear convolution ideals of $X$ are exactly $I_S=\{f\in X:c_n(f)=0\text{ whenever }n\notin S\}$ for $S\subseteq\Z$. Equivalently, $I_S=\overline{\spanC\{e_n:n\in S\}}$, with closure taken in $X$.
\end{theorem}
\begin{proof}
The bound $|c_n(f)|\leq M_1(f)$ makes each coefficient functional continuous, so $I_S$ is closed. Conversely, if $I$ is a closed ideal and $S=\{n:c_n(f)\ne0\text{ for some }f\in I\}$, then $e_n=c_n(f)^{-1}(f\conv e_n)\in I$ for each $n\in S$. For $g\in I_S$, its finite sums $\sum_{|n|\leq N}c_n(g)e_n$ belong to $I$ and converge uniformly to $g$ on compact subsets.
\end{proof}

\begin{theorem}[Finite-rank convolution operators]\label{thm:operators}
For $p\in X$, define $T_p f=p\conv f$. Then $T_p$ has finite rank if and only if $p$ is a finite Laurent polynomial. The nonzero rank-one convolution projections are exactly $T_{e_n}$.
\end{theorem}

\begin{proof}
Since $T_pe_n=c_n(p)e_n$, infinitely many nonzero coefficients give infinite rank. If the support is finite, the range is exactly its monomial span. Rank one and $T_p^2=T_p$ then require a single coefficient, equal to $1$.
\end{proof}

\medskip
\noindent The primitive idempotents also give the usual Laurent reconstruction.
For $E_N=\{-N,\ldots,N\}$ and every $f\in X$,
\[
 f=\lim_{N\to\infty} f\conv P_{E_N}
   =\sum_{n\in\Z}(f\conv e_n),
\]
with convergence uniform on compact subsets of $\Om$. This is the Laurent
expansion expressed through the coefficient projections.

\section*{Statements and Declarations}

\paragraph{Funding.}
No funding was received for this research.

\paragraph{Competing interests.}
The author declares that he has no competing interests.

\paragraph{Data availability.}
No datasets were generated or analyzed during the current study.

\paragraph{Declaration of Generative AI Use.}
Once the first draft of the manuscript was completed, generative AI tools (specifically ChatGPT) were employed strictly as supplementary aids to improve English readability, refine LaTeX formatting, identify potentially relevant references.
To be clear, all mathematical discoveries, conceptual developments, and scientific decisions---from the initial draft through the final revisions were made exclusively by the author.

\end{document}